\documentclass[10pt,oneside,reqno]{amsart}
\usepackage{amssymb,amsmath,amsthm,graphicx,verbatim,amsbsy}

\makeatletter
\@namedef{subjclassname@2020}{%
  \textup{2020} Mathematics Subject Classification}
\makeatother

\newcommand{\Z}{{\bf Z}}
\newcommand{\R}{{\bf R}}
\newcommand{\N}{{\bf N}}

\newcommand{\cl}{{\rm cl}_{L^2}\,}

\theoremstyle{plain}
\newtheorem{theorem}{Theorem}
\newtheorem{lemma}[theorem]{Lemma}

\newtheorem{corollary}[theorem]{Corollary}

\newtheorem{remark}[theorem]{Remark}

\theoremstyle{definition}

\def\st{{:}\,\ }

\title[Compactly supported orthogonal polynomial MRA]{Compactly supported, orthogonal, continuous piecewise polynomial multiresolution analysis}

\author[L~Fernandez]{Lidia~Fern\'andez}
\thanks{LF was supported in part by  PID2023-149117NB-I00, PID2024-155133NB-I00 and CEX 2020-001105-M, all funded by ``Ministerio de Ciencia, Innovaci\'on y Universidades'' MICIU/AEI/10.13039/501100011033 and ERDF/EU, Spain.}
\address{LF, Instituto de Matem\'aticas and Departamento de Matematica Aplicada, Universidad de Granada, Granada, Spain}
\email{lidiafr@ugr.es}

\author[J.~Geronimo]{Jeffrey~S.~Geronimo}
\address{JG, School of Mathematics, Georgia Institute of Technology, Atlanta, GA 30332--0160, USA}
\email{geronimo@math.gatech.edu}

\author[P.~Iliev]{Plamen~Iliev}
\thanks{PI gratefully acknowledges support by a grant from the Simons Foundation.}
\address{PI, School of Mathematics, Georgia Institute of Technology, Atlanta, GA 30332--0160, USA}
\email{iliev@math.gatech.edu}

\subjclass[2020]{41A15, 42C40, 33C20}

\keywords{Wavelets, Scaling Functions, Multiresolution Analysis, Hypergeometric functions, Piecewise polynomial.}

\theoremstyle{remark}
\newtheorem{rem}{Remark}
\def\hypergeom#1#2#3#4#5{{}_#1 F_{#2}\left({#3\atop#4}; \ #5\right)}

\begin{document}

\date{June 20, 2026}

\maketitle 

\begin{abstract}
We present explicit representations in terms of hypergeometric functions for the scaling functions in the $C^0$ orthogonal multiresolution analyses associated with piecewise continuous polynomials. Closed formulas for the  Mellin transform of these functions as well as their Fourier transforms are derived. Some new multiresolution analyses whose scaling functions have coefficients that are  rational numbers are introduced and discussed.
\end{abstract}

\section{Introduction}
Multiwavelets \cite{abcr}, \cite{gl}, \cite{hkm}, \cite{h},  have played a useful role in numerical analysis, functional analysis, and special functions \cite{bcc}, \cite{bcd},  \cite{dss}, \cite{dgha}, and \cite{ggl}.    The Alpert  multiresolution analysis \cite{abcr} was one of the first examples of a multiresolution analysis (MRA)  based on special functions (the Legendre polynomials). This MRA was further studied in \cite{gm}  and a compactly supported wavelet basis of generalized hypergeometric functions was constructed in \cite{gi}. Furthermore the connection with multiple orthogonal polynomials was considered in \cite{glv}. Since the multiwavelets constructed in the Alpert multiresolution analysis do not provide basis functions that are continuous on the real line
an alternative set of multiresolution analyses were developed in \cite{dgha}. These intertwining multiresolution analyses based on classes of ultraspherical polynomials allowed the construction of multiwavelets where the basis functions are at least continuous, however not all of the scaling functions were represented in terms of special functions. In this paper we begin by introducing a set of biorthogonal 'ramp' functions that have explicit representations in terms of Jacobi orthogonal polynomials.  Next we reconsider the $C^0$ case and show, for the first time, that all the scaling functions have simple formulas in terms of hypergeometric polynomials. These results help to elucidate the connection between special functions and piecewise polynomial multiresolution analyses.

The paper is organized as follows: In section 2 the notation is set and we review the elements of multiresolution analysis needed in the later sections.  In section 3 we introduce the orthogonal polynomials that will be used and also review formulas and recurrences in hypergeometric function theory. A set of biorthogonal 'ramp' functions are introduced that are given by  Jacobi polynomials which  simplifies some of the results developed in \cite{dgha}. We then investigate  multiresolution analyses that have $C^0$ generators and give explicit representations of the useful functions in this MRA in terms of hypergeometric polynomials. The Mellin transform of these functions is computed. These results are used to give explicit formulas  for the scaling functions computed in \cite{dgha} and also to develop new multiresolution analyses whose scaling function generators  are polynomials with rational coefficients  which may allow for efficient coding. 
We also note that the generators are symmetric or antisymmetric with respect to either $0$ or $\frac{1}{2}$. In section 5 we use the above results to explicitly compute the Fourier transform of the various functions used in these multiresolution analyses. In section 6 the $(n+1)\times (n+1)$ matrix coefficients in the refinement equation are considered and relations between their entries, based on the symmetries of the scaling functions discussed previously, are illuminated. An algorithm is outlined for the computation of the coefficients in the matrix equation giving the wavelets functions in terms of the scaling functions.  Finally following \cite{dgha} we describe using the symmetry of the wavelets and scaling functions how to easily modify these bases  for use on compact intervals.

\section{Preliminaries}
Let
$\phi_0,\dots,\phi_r$ be compactly supported $L^2$-functions, $\Phi_r=(\phi_0,\dots,\phi_r)^T, \ r\in\N$ and suppose that
\begin{align}\label{multiphi}
&V_0 = V_0(\Phi_r)=\cl{\rm span}\{\phi_i({\cdot}-j): i = 0,1,\dots,r,\ j\in\Z\}\ \text{and}\nonumber\\& V_p = \{f(2^p{\cdot}): f\in V_0, \ p\in\Z\} .
\end{align}  
Then
$V_0$ is called a  finitely generated shift invariant (FSI)
space.  The sequence $(V_p)_{p=-\infty}^{\infty}$ is called a
 multiresolution analysis generated by {$\phi_0,\dots,\phi_r$} if
(a) the spaces are nested, i.e. $\cdots\subset V_{-1}\subset V_0\subset V_1
\subset\cdots$, and (b) the generators $\phi_0,\dots,\phi_r$ and 
their integer translates form a Riesz basis for $V_0$ \cite{da}, \cite{ddr}, \cite{js}.  Because of (a)
and (b) above, we can write
$$
V_{j+1} = V_j \oplus W_j \quad\forall j\in\Z. 
$$
The space $W_0$ is called the wavelet space, and if
$\psi_0,\dots,\psi_r$ generate a shift-invariant basis for $W_0$, then
these functions are called  wavelet functions.
If, in addition, $\phi_0,\dots,\phi_r$ 
and their integer translates form an orthogonal basis for $V_0$,
then $(V_p)_{p=-\infty}^{\infty}$ is called an  orthogonal MRA.  In \cite{dgh} it is shown that  by adjusting $r$ we can always assume that $\phi_j$, $j = 0 , \ldots, r$  can be
chosen so that they are {\it minimally supported} on $[-1,1]$; i.e., each
scaling function has support in $[-1,1]$, and the nonzero restrictions of
the scaling functions and their integer translates to $[0,1]$ are linearly
independent.  Then (a) and (b) imply
that $\Phi_r$ satisfies the matrix dilation equation
\begin{equation}\label{refl}
\Phi_r(t) = \sum_{i=-2}^1 C_i (r)\Phi_r(2t-i). 
\end{equation}
Sometimes we will also use  $\Phi_r$  to denote the set of its entries $\phi_j$.  Following \cite{dgha} we divide $V_0\chi_{[0,1]}$ into the following subspaces:
\begin{itemize}
\item $ {\mathcal A}_0={\rm span}\{\phi\in\Phi_r\st\ {\rm supp}
  \;\phi \subset [0,1]\} = \{g\in V_0\st\ {\rm supp}\,g\subset [0,1]\}$
\item $ \mathcal{C}_0 = {\rm span}
  \{\phi({\cdot})\chi_{[0,1]}\st\ \phi\in\Phi_r\}\ominus \mathcal{A}_0$
\item
  $\mathcal{C}_1={\rm span}\{\phi ({\cdot}-1)\chi_{[0,1]}\st\ \phi\in\Phi_r\}\ominus \mathcal{A}_0$.
\end{itemize}
Since the functions in $\mathcal{A}_0$ are orthogonal to their integer
translates  and a basis of $\mathcal{A}_0$ may be made orthogonal by Gram-Schmidt, it is not
difficult to see \cite{dgha},
\begin{theorem}\label{theoremone}
$(V_p)_{p=-\infty}^{\infty}$ is an orthogonal MRA iff $\mathcal{C}_0\perp \mathcal{C}_1$.
\end{theorem}

The MRAs we will study are those associated
with piecewise polynomial splines.  Let $S_k^n$ be the space of
polynomial splines of degree $n$ with $C^k$ knots at the integers,
and set $V_0^{n,k} = S_k^n \cap L^2(\R)$.  With  $V_p^{n,k}$ as above it
is known  that the integer shifts of the B-spline in $S_k^n$ form a Riesz basis for $V_0^{n,k} $ so that  these spaces give rise to  multiresolution analyses \cite{deb}, \cite{dgh}, \cite{gl}. If $k=-1$ (the spaces of piecewise polynomials of degree $n$ with integer knots)
then these multiresolution analyses are called the Alpert MRAs. 
Note that the dimension of  $V^{n,k}_0\chi_{[0,1]}$ is $n+1$ so the dimension of  $
\mathcal{A}^{n,k}_0$ is $n-2k-1$. Thus the dimensions of $\mathcal{C}^{n,k}_0$ and $\mathcal{C}^{n,k}_1$ are each $k+1$.

In order to find a compactly supported continuous orthogonal basis  for the above multiresolution analyses the theory of  intertwining MRAs  which were introduced in \cite{dgh} and \cite{dgha} will be used. We first develop useful bases for the above spaces. 

\section{Jacobi Polynomials}

We now introduce  some hypergeometric polynomials and identities that will be useful for the calculations below. Let $\{p^{\alpha,\beta}_n(x)\}_{n=0}^{\infty}, \ \alpha,\beta>-1$ be the monic Jacobi polynomials which are uniquely determined by 
$$
p_n^{\alpha,\beta}(x)=x^n + {\rm lower\ order\ terms\ in}\  x ,
$$
and
$$
\int_{-1}^1 p_n^{\alpha,\beta}(x) x^i (1-x)^{\alpha}(1+x)^{\beta}dx=0, \ 0\le i<n.
$$
These polynomials have the following representation in terms of
hypergeometric functions.
\begin{equation}\label{hyperone}
p_n^{\alpha,\beta}(t) = \frac{2^n(\alpha+1)_n}{(n+\alpha+\beta+1)_n}
  \,\hypergeom21{-n,\ n+\alpha+\beta+1}{\alpha+1}{\frac{1-t}{2}},
\end{equation}
where formally,
$$
\hypergeom{p}{q}{a_1,\ \dots\ a_p}{b_1,\ \dots\ b_q}{t}
  = \sum_{i=0}^{\infty}\frac{(a_1)_i\dots(a_p)_i}{(b_1)_i\dots(b_q)_i(1)_i}t^i
$$ 
with $(a)_0=1$ and $(a)_i = (a)(a+1)\ldots(a+i-1)$ for $i>0$.  Since one of the numerator 
parameters in the definition of $p_n^{\alpha,\beta}$ is a nonpositive integer the series has finitely many terms. When $\alpha=\beta=\lambda-\frac{1}{2}$ the Jacobi polynomials are called the ultraspherical polynomials, $p^{\lambda}_n$, and they have an alternate representation (see \cite{sz}, equations 4.7.30 and 4.7.31) given by,
\begin{equation}\label{altdefe}
p_{2n+\epsilon}^{\lambda}(t) = \frac{(-1)^{n}(\frac{1}{2}+\epsilon)_n}{(n+\lambda+\epsilon)_n} t^{\epsilon}
  \,\hypergeom21{-n,\ n+\lambda+\epsilon}{\frac{1}{2}+\epsilon}{t^2}, 
\end{equation}
for $\epsilon\in\{0,1\}$.
Some useful formulas are (\cite{aar},\cite{ba}):
\begin{itemize}
\item{Chu-Vandermonde formula},
$$
\hypergeom21{-n,a}{b}{1}=\frac{(b-a)_n}{(b)_n}
$$
and 
\item{Pfaff-Saalsch\"utz formula},
$$
\hypergeom32{-n,a,b}{c,-n+a+b-c+1}{1}=\frac{(c-a)_n(c-b)_n}{(c)_n(c-a-b)_n}.
$$
\end{itemize}
We also need the recurrence relations,
\begin{equation}\label{rec1}
b\ \hypergeom43{a,b+1,c,d}{e,f,g}{z}=c\ \hypergeom43{a,b,c+1,d}{e,f,g}{z}+(b-c)\ \hypergeom43{a,b,c,d}{e,f,g}{z},
\end{equation}
\begin{align}\label{rec2}
&(b-a)cdz\hypergeom43{a+1,b+1,c+1,d+1}{e+1,f+1,g+1}{z}\nonumber\\&=efg\left(\hypergeom43{a+1,b,c,d}{e,f,g}{z}-\hypergeom43{a,b+1,c,d}{e,f,g}{z}\right),
\end{align}
\begin{equation}\label{rec3}
bcz\hypergeom32{a+1,b+1,c+1}{d+1,e+1}{z}=de\left(\hypergeom32{a+1,b,c}{d,e}{z}-\hypergeom32{a,b,c}{d,e}{z}\right),
\end{equation}
\begin{align}\label{rec4}
&f\ \hypergeom43{a,b,c,d}{e+1,f,g}{z}\nonumber\\&=e\ \hypergeom43{a,b,c,d}{e,f+1,g}{z}-(e-f)\hypergeom43{a,b,c,d}{e+1,f+1,g}{z},
\end{align}
\begin{equation}\label{rec5}
d\ \hypergeom32{a,b,c}{d,e}{z}=a\ \hypergeom32{a+1,b,c}{d+1,e}{z}-(a-d)\hypergeom32{a,b,c}{d+1,e}{z},
\end{equation}
which can be checked formally term by term.
Below is a Lemma that will be used often:
\begin{lemma}\label{hyp1}
With $-n+a=d$,
\begin{equation}\label{hyp2}
\hypergeom32{-n+1,a+1,b+1}{d+1,b+2}{1}= (-1)^{n+1}\frac{(n-1)!(b+1)}{(d+1)_n}\left(1+(-1)^{n+1}\frac{(d-b)_n}{(b+1)_n}\right),
\end{equation}
while for $-n+a-1=d$,
\begin{align}\label{hyp3}
&\hypergeom32{-n+1,a+1,b+1}{d+1,b+3}{1}\nonumber\\&=\frac{d(b+2)}{an}\frac{(d-1)(b+1)}{(n+1)(a-1)}\left(\frac{(d-a)_{n+1}}{(d-1)_{n+1}}-\frac{(d-a)_{n+1}(d-b-1)_{n+1}}{(d-1)_{n+1}(d-a-b)_{n+1}}\right)\nonumber\\&+\frac{d(b+2)}{an}\frac{(d-a)_n(d-b)_n}{(d)_n(d-a-b)_n}.
\end{align}
\end{lemma}
\begin{proof}
Equation~\eqref{hyp2} follows from the application of the Chu-Vandermonde and the Pfaff-Saalsch\" utz formulas to equation~\eqref{rec3}. Equation~\eqref{hyp3} follows by using equation~\eqref{rec3} twice then using the Chu-Vandermonde and the Pfaff-Saalsch\" utz formulas.
\end{proof}
\begin{lemma}\label{abp1}
For $\alpha$ a nonnegative integer less than or equal to $n$ and $\beta>-1$,
\begin{align*}
I&=\int_{-1}^1 (1-t)^{\alpha}(1+t)^{\beta}p_{n-\alpha}^{\alpha,\beta+1}(t)p_{n-\alpha}^{\alpha,\beta+1}(-t)dt\\&=2^{2n+\beta-\alpha+1}\frac{(\alpha+1)_{n-\alpha}\Gamma(\beta+1)n!(-n+\alpha)_{n-\alpha}}{(n+\beta+2)_{n-\alpha}\Gamma(2n-\alpha+\beta+2)}.
\end{align*}
\begin{proof}
The use of \eqref{hyperone} and the orthogonality properties of the polynomials in the above integral yields
\begin{equation*}
I=\left(\frac{2^{n-\alpha}(\alpha+1)_{n-\alpha}}{(n+\beta+2)_{n-\alpha}}\right)^2\ \sum_{k=0}^{n-\alpha}\frac{(-n+\alpha)_k(n+\beta+2)_k}{(1)_k(\alpha+1)_k2^k}\int_{-1}^1(1-t)^{\alpha+k}(1+t)^{\beta}dt.
\end{equation*}
The integral gives
\begin{equation*}
\int_{-1}^1(1-t)^{\alpha+k}(1+t)^{\beta}dt=2^{\alpha+\beta+k+1}\frac{(\alpha+1)_k(\alpha)!\Gamma(\beta+1)}{(\alpha+\beta+2)_k\Gamma(\alpha+\beta+2)},
\end{equation*}
so that the terms containing $k$ give
\begin{equation*}
\sum_{k=0}^{n-\alpha}\frac{(-n+\alpha)_k(n+\beta+2)_k}{(1)_k(\alpha+\beta+2)_k}=\frac{(-n+\alpha)_{n-\alpha}}{(\alpha+\beta+2)_{n-\alpha}},
\end{equation*}
by the Chu-Vandermonde formula. This leads to the result.
\end{proof}
\end{lemma}
It easily follows (see \cite{dgha} Lemma 2.1) 
\begin{lemma}
For $n\ge 2k+2$, an orthogonal basis  for the functions in $\mathcal{A}^{n,k}_0$ supported in $[0,1]$ is $\{t^{k+1}(1-t)^{k+1}
p_i^{2k+\frac{5}{2}}(2t-1), \ i = 0,\dots,n-2k-2\}$, where each
$p_i^{2k+\frac{5}{2}}$  is a monic ultraspherical polynomial of degree $i$.
\end{lemma}

For later use we introduce the set of functions
\begin{equation}\label{phip}
\phi^k_i(t)=(1-t^2)^{k+1}p_{i-2k-2}^{2k+\frac{5}{2}}(t)
\end{equation}
 for $i\ge 2k+2$ which, from the above lemma, are orthogonal on $[-1,1]$ and for $i=2k+2, \ldots , n$ form an orthogonal basis  for $\mathcal{A}^{n,k}_0(\frac{\cdot +1}{2})$. We  set $\phi^k_i(t)=0$ for $i<2k+2$.

We now introduce the following  functions
\begin{equation}\label{rnki}
\tilde r^{n,k}_i(t)=(1-t)^i(1+t)^{k+1}p^{i+k+1,i+k+2}_{n-k-i-1}(t),\quad i=0,\ldots,k,\ n\ge2k+2
\end{equation}
and
\begin{equation}\label{lnki}
\tilde l^{n,k}_i(t)=\tilde r^{n,k}_i(-t).
\end{equation}
These functions have the following orthogonality properties,
\begin{theorem}\label{ramp}
For $n\ge 2k+2$ the set $\{\tilde r^{n,k}_i\}_{i=0}^{k}$ is orthogonal to $\mathcal{A}^{n,k}_0(\frac{\cdot + 1}{2})$ and biorthogonal  to the set  $\{\tilde l^{n,k}_i\}_{i=0}^{k} $  , hence it is a basis for $\mathcal{C}^{n,k}_0(\frac{\cdot+1}{2})$. Furthermore
\begin{equation}\label{ppnk}
\int_{-1}^1 \tilde r^{n,k}_i(t)\tilde l^{n,k}_i(t)dt=2^{2n+1}\frac{(n!)^2(-n+k+i+1)_{n-k-i-1}}{(n+k+i+3)_{n-k-i-1}(2n+1)!}.
\end{equation}
\end{theorem}
Note that the above functions are linear combinations of the 'ramp' functions  $r^{n,k}_i$ introduced in  \cite[section 2]{dgha}.
\begin{proof}
The first part of the Theorem follows since
\begin{align*}
&\int_{-1}^1\tilde r^{n,k}_i(t)\phi^k_j(t)dt\\&=\int_{-1}^1(1-t)^{i+k+1}(1+t)^{i+k+2}p^{i+k+1,i+k+2}_{n-k-i-1}(t)(1+t)^{k-i}p_{j-2k-2}^{2k+\frac{5}{2}}(t)dt=0,
\end{align*}
for $k+i+2\le 2k+2\le j\le n$  since $i\le k$. To see the second part of the Theorem observe that for $k$ fixed
\begin{align*}
I=&\int_{-1}^1 \tilde r^{n,k}_i(t) \tilde l^{n,k}_j(t)dt=\int_{-1}^1 \tilde r^{n,k}_i(t) \tilde r^{n,k}_j(-t)dt\\&=\int_{-1}^1 (1-t)^{i+k+1}(1+t)^{j+k+1}p^{i+k+1,i+k+2}_{n-k-i-1}(t) p^{j+k+1,j+k+2}_{n-k-j-1}(-t)dt.
\end{align*}
If $i<j\le n$ the above integral can be written as
\begin{equation*}
I=\int_{-1}^1 (1-t)^{i+k+1}(1+t)^{i+k+2}p^{i+k+1,i+k+2}_{n-k-i-1}(t) (1+t)^{j-i-1}p^{j+k+1,j+k+2}_{n-k-j-1}(-t)dt=0.
\end{equation*}
For $i>j$ let $t\to -t$  in the above integral and interchange the roles of $i$ and $j$ to see that it is equal to zero. The value of the integral in equation~\eqref{ppnk} follows from Lemma~\ref{abp1} with $\alpha=\beta=i+k+1$.
\end{proof}
\section{The $C^0$ case}
In this section we prove one of the main results of the paper which is to give simple hypergeometric representations to all of the scaling functions in the continuous case. In order to produce an orthogonal MRA with the functions above we will need the theory of intertwining MRA which was discussed in \cite{dgh}.  Let $P^{n,0}$ be the orthogonal projection on $\mathcal{A}^{n,0}_0(\frac{\cdot+1}{2})$ and $r^{n,0}_0$ be given as in \cite{dgha}, i.e. $r^{n,0}_0(t)=(I-P^{n,0})r^0_0(t)$ with $r^0_0=1+t, \ t\in[-1,1]$. Also set $l_0^0(t)=r_0^0(-t)$ so that $l_0^{n,0}(t)=r_0^{n,0}(-t)$. For ease of notation we will drop the second superscript in these functions and in the spaces. We now have

\begin{lemma}\label{rhatr}
\begin{equation}\label{rhr}
r^n_0(t)=\frac{2(2n+1)!!}{(n+2)!}\tilde r^n_0(t),
\end{equation}
thus
\begin{equation}\label{roro}
\langle r^n_0,l^n_0\rangle=\frac{(-1)^{n+1}8}{(n+2)(n+1)n}.
\end{equation}
Furthermore
\begin{align*}
r_0^n(t)&=(-1)^n(p_n^{\frac{3}{2}}(t)/p_n^{\frac{3}{2}}(-1)-p_{n-1}^{\frac{3}{2}}(t)/p^{\frac{3}{2}}_{n-1}(-1))\\&=\frac{(-1)^n}{{p_n^{\frac{3}{2}}(-1)p_{n-1}^{\frac{3}{2}}(-1)}}\left |\begin{matrix}p^{\frac{3}{2}}_n(t)& p^{\frac{3}{2}}_{n-1}(t)\\ p^{\frac{3}{2}}_n(-1)& p^{\frac{3}{2}}_{n-1}(-1)\end{matrix}\right| .
\end{align*}
\end{lemma}
\begin{proof}
Since $\mathcal{C}^{n,0}_0(\frac{\cdot+1}{2})$ is one dimensional $r^{n}_0$ and $\tilde r^{n}_0(t)$ are related by a multiplicative  constant which can be found by setting $t=1$.
The second formula follows from equations~\eqref{ppnk} and \eqref{rhr} while the third formula follows from \cite[4.5.4]{sz}  taking into account the normalization used.
\end{proof}

We now give representations for some functions introduced in \cite{dgha} used to build an intertwining MRA. Let $\bar u_m\in C^0(\R)$ be given as
$$
\bar u_m(t) = \left\{
  \begin{matrix}1 - |t|\,t^{m-1}
      & {\rm if}\ t\in[-1,1]\hfill\\
    0 & {\rm otherwise}      \hfill\end{matrix}\right.
$$
for odd $m\ge 1$, and
$$
\bar u_m(t) = \left\{\begin{matrix}
  t - |t|\,t^{m-1} 
    & {\rm if}\ t\in[-1,1] \hfill\\
  0 & {\rm otherwise}      \hfill\end{matrix}\right.
$$
for even $m\ge 2$. For $n\ge 2$ let
\begin{equation}\label{unm}
u_{n,n-m}=(I-P^n)\bar u_{m}.
\end{equation}
 
Note that due to the parity of $\bar u_m$ and $\phi_i$, $\langle \bar u_m,\phi_i\rangle=0$ if $m$ and $i$ have the same parity.
We now prove
\begin{theorem}\label{rep1}
For $t\in [-1,1]$ and $\epsilon\in\{0,1\}$ the functions $u_{n,m}$ have the following representations
$$
u_{2n+\epsilon,2n-2m-1}(t)=u_{2n+1+\epsilon,2n-2m}(t)=-|t|t^{2m+\epsilon}+f^\epsilon_{m,n}(t),
$$
where 
\begin{align}\label{fe}
f^{\epsilon}_{m,n}(t)&= t^{\epsilon}\frac{(\frac{1}{2}+\epsilon)_{n} (n^2+(\frac{3}{2}+\epsilon)n+m+1+\epsilon)(-m+\frac{1}{2})_{n}}{(m+1+\epsilon)_{n+1}(n+1)!}\nonumber\\&\times 
\hypergeom43{-n,-n^2-(\frac{3}{2}+\epsilon)n-m-\epsilon,-m-\frac{1}{2},n+\frac{3}{2}+\epsilon}{\frac{1}{2}+\epsilon,-n^2-(\frac{3}{2}+\epsilon)n-m-1-\epsilon,-m+\frac{1}{2}}{t^2}.
\end{align}

\end{theorem}
Some preliminary results will be needed before proving this Theorem. We begin with,
\begin{lemma}\label{fefo}
The function $f^{\epsilon}_{m,n}$ has the representation
$$
 f^{\epsilon}_{m,n}(t)=f^{\epsilon}_{1,m,n}(t)+f^{\epsilon}_{2,m,n}(t),
$$ 
where
\begin{align}\label{f1e}
&f^{\epsilon}_{1,m,n}(t)\nonumber\\&=t^{\epsilon}\frac{(\frac{1}{2}+\epsilon)_{n+1}(-m+\frac{1}{2})_n}{(m+1+\epsilon)_{n+1} n!}\hypergeom32{-n,-m-\frac{1}{2},n+\frac{3}{2}+\epsilon}{\frac{1}{2}+\epsilon,-m+\frac{1}{2}}{t^2}
\end{align}
and
\begin{align}\label{f2e}
f^{\epsilon}_{2,m,n}(t)&=-t^{\epsilon}\frac{(\frac{1}{2}+\epsilon)_n(-m-\frac{1}{2})_{n+1}}{(m+1+\epsilon)_{n+1}(n+1)!}\hypergeom21{-n,n+\frac{3}{2}+\epsilon}{\frac{1}{2}+\epsilon}{t^2}\nonumber\\&=\frac{(-1)^{n+1}(-m-\frac{1}{2})_{n+1}}{(m+1+\epsilon)_{n+1}}\frac{p_{2n+\epsilon}^{\frac{3}{2}}(t)}{p_{2n+\epsilon}^{\frac{3}{2}}(1)}.
\end{align}
\end{lemma}
\begin{proof}
Use  equation~\eqref{rec1} with
$$
a=-n, \ b=f=-n^2-(\frac{3}{2}+\epsilon)n-m-\epsilon-1,\ g-1=c=-m-\frac{1}{2},\ d=n+\frac{3}{2}+\epsilon, \ e=\frac{1}{2}+\epsilon.
$$ 
\end{proof}
The Mellin transform of the above function on $[0,1]$ is given by
\begin{lemma}\label{mt} For real $z>0$,
\begin{align}\label{mtfe}
\check f^{\epsilon}_{m,n}(z)&=\int_0^1 t^{z-1} f^{\epsilon}_{m,n}(t)dt\nonumber\\&=\frac{1}{z+2m+1+\epsilon}+(-1)^{n+1}\frac{(-m-\frac{1}{2})_{n+1}}{(m+1+\epsilon)_{n+1}(n+1)(2n+1+2\epsilon)}\nonumber\\&
+(-1)^n\frac{(\frac{z+\epsilon}{2}+m+\frac{1}{2}+(n+1)(n+\frac{1}{2}+\epsilon))}{(\frac{z+\epsilon}{2}+m+\frac{1}{2})(n+1)(2n+1+2\epsilon)}\frac{(-m-\frac{1}{2})_{n+1}(-\frac{z-\epsilon}{2}+\frac{1}{2})_{n+1}}{(m+1+\epsilon)_{n+1}(-n-\frac{z+\epsilon}{2})_{n+1}}.
\end{align}
\end{lemma}
\begin{proof}
Equations~\eqref{f1e} and \eqref{f2e} yield
\begin{align*}
\check f^{\epsilon}_{m,n}(z)&=\frac{(\frac{1}{2}+\epsilon)_n(n+\frac{1}{2}+\epsilon)(-m+\frac{1}{2})_n}{(m+1+\epsilon)_{n+1} n!(z+\epsilon)}\hypergeom43{-n,-m-\frac{1}{2}, n+\frac{3}{2}+\epsilon, \frac{z+\epsilon}{2}}{\frac{1}{2}+\epsilon,-m+\frac{1}{2},\frac{z+\epsilon}{2}+1}{1}\nonumber\\&-\frac{(\frac{1}{2}+\epsilon)_n(-m-\frac{1}{2})_{n+1}}{(m+1+\epsilon)_{n+1}(n+1)!(z+\epsilon)}\hypergeom32{-n,n+\frac{3}{2}+\epsilon,\frac{z+\epsilon}{2}}{\frac{1}{2}+\epsilon,\frac{z+\epsilon}{2}+1}{1}.
\end{align*}
With the use of the formulas~\eqref{rec2} with $a=-m-\frac{3}{2}$ and $b=\frac{z+\epsilon}{2}-1$ and \eqref{rec3} with $a+1=\frac{z+\epsilon}{2}$ yields
\begin{align*}
\check f^{\epsilon}_{m,n}(z)&=-\frac{(-\frac{1}{2}+\epsilon)_{n+1}(-m-\frac{1}{2})_{n+1}}{2(m+1+\epsilon)_{n+1} (n+1)!(\frac{z+\epsilon}{2}+m+\frac{1}{2})}\\&\left(\hypergeom32{-n-1,n+\frac{1}{2}+\epsilon,\frac{z+\epsilon}{2}-1}{-\frac{1}{2}+\epsilon,\frac{z+\epsilon}{2}}{1}-\hypergeom32{-n-1,n+\frac{1}{2}+\epsilon,-m-\frac{3}{2}}{-\frac{1}{2}+\epsilon,-m-\frac{1}{2}}{1}\right)\nonumber\\&+\frac{(-\frac{1}{2}+\epsilon)_{n+1}(-m-\frac{1}{2})_{n+1}}{2(m+1+\epsilon)_{n+1}(n+1)!(n+1)(n+\frac{1}{2}+\epsilon)}\\&\left(\hypergeom21{-n-1,n+\frac{1}{2}+\epsilon}{-\frac{1}{2}+\epsilon}{1}-\hypergeom32{-n-1,n+\frac{1}{2}+\epsilon,\frac{z+\epsilon}{2}-1}{-\frac{1}{2}+\epsilon,\frac{z+\epsilon}{2}}{1}\right).
\end{align*}
Combining like terms then using the Chu-Vandermonde formula and the Pfaff-Saalsch\"utz formula yields the result after simplification.
\end{proof}
\begin{proof}[Proof of Theorem~\ref{rep1}]
Note that from its definition $u_{2n+1+\epsilon,2n-2m}+|t|t^{2m+\epsilon}$ is a polynomial of degree at most $2n+1+\epsilon$ on $[-1,1]$. This polynomial is characterized by the $2n+\epsilon$ orthogonality equations
$$
\int_{-1}^1 u_{2n+1+\epsilon,2n-2m}(t)t^i(1-t^2)dt=0, \ i=0,\ldots, 2n-1+\epsilon
$$
and the evaluations $u_{2n+1+\epsilon,2n-2m}(\pm1)=0$.
Set $u_{\epsilon,m,n}(t)= -|t|t^{2m+\epsilon}+f^\epsilon_{m,n}(t)$ so that
\begin{align*}
I_{n,i}&=\int_{-1}^1 u_{\epsilon,m,n}(t)t^i(1-t^2)dt=-\int_{-1}^1 |t|t^{2m+\epsilon}t^i(1-t^2)dt+\int_{-1}^1 f^{\epsilon}_{m,n}(t)t^i(1-t^2)dt\\&=
-(1+(-1)^{i+\epsilon})\left(\frac{1}{2m+i+2+\epsilon}-\frac{1}{2m+i+4+\epsilon}\right)\\&+(1+(-1)^{i+\epsilon})\int_0^1f^{\epsilon}_{m,n}(t)t^i(1-t^2)dt.
\end{align*}
We only need to consider the case when $i$ and $\epsilon$ have the same parity. With the use of the Mellin transform  with $z=i+1$ and $z=i+3$ we find only the difference in the third term
survives which has the vanishing factor  $(-\frac{z-\epsilon}{2}+\frac{1}{2})_{n+1}$. Thus $I_{n,i}=0,\ i=0,\ldots, 2n-1$. 
Since $u_{\epsilon,m,n}$ is either symmetric or antisymmetric in $t$ we need only show that $u_{\epsilon,m,n}(1)=0$ or 
\begin{equation}\label{umnone}
1=f^{\epsilon}_{1,m,n}(1)+f^{\epsilon}_{2,m,n}(1).
\end{equation}
Lemma~\ref{fefo} gives
$$
f^{\epsilon}_{2,m,n}(1)=\frac{(-1)^{n+1}(-m-1/2)_{n+1}}{(m+1+\epsilon)_{n+1}}
$$
and 
$$
f^{\epsilon}_{1,m,n}(1)=\frac{(1/2+\epsilon)_{n+1}(-m+1/2)_n}{(m+1+\epsilon)_{n+1} n!}\hypergeom32{-n,-m-1/2,n+3/2+\epsilon}{1/2+\epsilon,-m+1/2}{1}.
$$
Lemma~\ref{hyp1}  shows that
\begin{equation*}
\hypergeom32{-n,-m-1/2,n+3/2+\epsilon}{1/2+\epsilon,-m+1/2}{1}=\frac{(-1)^nn!(-m-\frac{1}{2})}{(\frac{1}{2}+\epsilon)_{n+1}}\bigg(1+(-1)^n\frac{(m+1+\epsilon)_{n+1}}{(-m-\frac{1}{2})_{n+1}}\bigg).
\end{equation*} 
Thus
$$
f^{\epsilon}_{1,m,n}(1)=-\frac{(-1)^{n+1}(-m-\frac{1}{2})_{n+1}}{(m+1+\epsilon)_{n+1}}+1,
$$
so that  $f^{\epsilon}_{1,m,n}(1)+f^{\epsilon}_{2,m,n}(1)=1$ and equation~ \eqref{umnone} is satisfied. The theorem now follows since $\phi_{2n+1+\epsilon}$ and $(1-|t|t^{2m})t^{\epsilon}$ have opposite parities on $[-1,1]$.
\end{proof}
With the above formulas we evaluate some integrals that will be used below  (see \cite{dgha} equation 5.4).
\begin{lemma}\label{ru}
With $\epsilon\in\{0,1\}$
\begin{align}\label{rue}
\langle r^{2n+1+\epsilon}_0, u_{2n+1+\epsilon,2n-2m}\rangle&=\int_{-1}^1 r^{2n+1+\epsilon}_0(t) u_{2n+1+\epsilon,2n-2m}(t)dt\nonumber\\&=(-1)^{n+1}\frac{(-m-\frac{1}{2})_{n+1}}{(m+1+\epsilon)_{n+1}(n+\frac{1}{2}+\epsilon)(n+1)},
\end{align}
and
\begin{align}\label{uuev}
&\langle u_{2n+1+\epsilon,2n-2m}, u_{2n+1+\epsilon,2n-2m_1}\rangle=\int_{-1}^1 u_{2n+1+\epsilon,2n-2m}(t) u_{2n+1+\epsilon,2n-2m_1}(t)dt\nonumber\\&=
\frac{(-m_1-\frac{1}{2})_{n+1}(-m-\frac{1}{2})_{n+1}(m_1+m+\frac{3}{2}+\epsilon+(n+1)(n+\frac{1}{2}+\epsilon))}{(m_1+1+\epsilon)_{n+1}(m+1+\epsilon)_{n+1}(m_1+m+\frac{3}{2}+\epsilon)(n+1)(n+\frac{1}{2}+\epsilon)}.
\end{align}
\end{lemma}
\begin{proof}
The use of equation~\eqref{rhr} and Theorem~\ref{rep1} gives
\begin{align*}
&\int_{-1}^1 r^{2n+1+\epsilon}_0(t) u_{2n+1+\epsilon,2n-2m}(t)dt=\int_{-1}^1 (1+t) u_{2n+1+\epsilon,2n-2m}(t)dt\\&=-\frac{1}{m+1+\epsilon}+\int_{-1}^1(1+t)  f^{\epsilon}_{m,n}(t)dt.
\end{align*}
Formula~\eqref{mtfe} with $z=1$  and $z=2$ now gives the result.
To obtain the second formula note that
\begin{align*}
&\langle u_{2n+1+\epsilon,2n-2m}, u_{2n+1+\epsilon,2n-2m_1}\rangle\\&=\int_{-1}^1(t^{\epsilon}-|t|t^{2m+\epsilon})(-|t|t^{2m_1+\epsilon}+f^{\epsilon}_{m_1,n}(t))dt\\&=-\frac{1}{m_1+1+\epsilon}+\frac{2}{2m_1+2m+3+2\epsilon}\\&+2\int_0^1(t^{\epsilon}- t^{2m+1+\epsilon})f^{\epsilon}_{m_1,n}(t)dt.
\end{align*}
The result now follows from formula~\eqref{mtfe} with $z=1+\epsilon$ and $z=2m+2+\epsilon$.
\end{proof}
\begin{corollary}\label{corrn0un0}
The above formulas yield for $n>2m$
\begin{equation}\label{rn0un0}
\langle r^n_0, u_{n,2m}\rangle=\frac{(-1)^m(n-1)!(n-2m)!(2m)!}{2^{n-2}(n+1)!(n-m)!m!}=(-1)^{n+1}\langle l^n_0,u_{n,2m}\rangle
\end{equation}
and for $n\ge2\text{max}\{m,m_1\}+1$
\begin{align}\label{umum1}
&\langle u_{n,2m},u_{n,2m_1}\rangle\nonumber\\&=\frac{(-1)^{m+m_1}(n-2m)!(n-2m_1)!(2m)!(2m_1)!(n-1)!(n^2+5n+2-4(m+m_1))}{2^{2n-1}m!m_1!(n-m)!(n-m_1)!(n+1)!(2n+1-2(m+m_1))}\nonumber\\&=\frac{\langle r^n_0,u_{n,2m}\rangle\langle r^n_0,u_{n,2m_1}\rangle(n+1)!(n^2+5n+2-4(m+m_1))}{8(n-1)!(2n+1-2(m+m_1))}.
\end{align}
\end{corollary}
\begin{rem}
Formula~\eqref{umum1} corrects a typo in equation~(5.8) in \cite{dgha}.
\end{rem}
\begin{proof}
From  equation~\eqref{rue} we have with $m\to n-m$
$$
\langle r^{2n+1+\epsilon}_0,u_{2n+1+\epsilon,2m}\rangle=\frac{(-1)^{n+1}(-n+m-\frac{1}{2})_{n+1}}{(n-m+1+\epsilon)_{n+1}(n+\frac{1}{2}+\epsilon)(n+1)}.
$$
Since for $n>m$ and $\epsilon\in\{0,1\}$
$$
(n-m+1+\epsilon)_{n+1}(n+\frac{1}{2}+\epsilon)(n+1)=\frac{(2n-m+1+\epsilon)!(2n+2+\epsilon)!}{4(n-m+\epsilon)!(2n+\epsilon)!},
$$
and 
\begin{align*}
(-n+m-\frac{1}{2})_{n+1}&=(-n+m-\frac{1}{2})\cdots(-\frac{1}{2})(\frac{1}{2})\cdots(m-\frac{1}{2})\\&=(-1)^{n-m+1}\frac{(2n-2m+1)\cdots(1)(2m-1)\cdots(1)}{2^{n-m+1} 2^m}\\&=(-1)^{n-m+1}\frac{(2n-2m+1)!(2m)!}{2^{2n+1}(n-m)!m!},
\end{align*}
the first equation follows. The second equality in equation~\eqref{rn0un0} follows from the parity of $u_{n,2m}$ since 
\begin{align*}
\langle l^n_0,u_{n,2m}\rangle&=\langle (1-t),(I-P^n)\bar u_{n-2m}\rangle=(-1)^{n+1}\langle (1+t),(I-P^n)\bar u_{n-2m}\rangle\\&=(-1)^{n+1}\langle r^n_0,u_{n,2m}\rangle.
\end{align*}
To show the second formula let $m\to n-m$ and $m_1\to n-m_1$ in equation~\eqref{uuev} then use the  formulas just above.
\end{proof}

Let $w_n$ be given by equation~(5.2) in \cite{dgha} i.e.
\begin{equation}\label{symw}
w_n=\alpha(n) u_{n,0} + u_{n,2}
\end{equation}  
with
$$
\alpha(n)=2\frac{\frac{(n-2)(2n+1)}{n-1}-2\sqrt{3}\sqrt{\frac{(2n+1)(n+1)}{(2n-3)(n-1)}}}{n(2n-1)},
$$
and set 
\begin{equation}\label{prl}
q_{l_0^n}(-t)= q_{r_0^n}(t)=(I-P_{w_n})r_0^n(t)
\end{equation}
where $P_{w_n}$ is the orthogonal projection onto the space spanned by $w_n$.
Then we have:
\begin{theorem} The above functions have the following explicit formulas,
\begin{align*}
w_{2n+1}(t)=-(1+\alpha(2n+1)t^2)|t|t^{2n-2}+\alpha(2n+1)f^0_{n,n}(t)+f^0_{n-1,n}(t),
\end{align*}
while
\begin{align*}
w_{2n+2}(t)=-(1+\alpha(2n+2)t^2)|t|t^{2n-1}+\alpha(2n+2)f^1_{n,n}(t)+f^1
_{n-1,n}(t) 
\end{align*}
and
\begin{equation}\label{ortoramp}
q_{l_0^n}(-t)= q_{r_0^n}(t)=\frac{2(2n+1)!!}{(n+2)!}(1+t)p^{1,2}_{n-1}(t)-\frac{\langle w_n,r_0^n\rangle}{\langle w_n, w_n\rangle}w_n(t),\ -1\le t\le 1, 
\end{equation}
$$
\frac{\langle w_n,r_0^n\rangle}{\langle w_n, w_n\rangle}=2^{n-2}\frac{(n-1)\left((n+1)(2n-3)-\sqrt{3}\sqrt{(2n+1)(2n-3)(n-1)(n+1)}\right)}{(n+1)(n+2)}.
$$
\end{theorem}
\begin{proof}
The formula for $\frac{\langle w_n,r_0^n\rangle}{\langle w_n, w_n\rangle}$ above follows from  Corollary~\ref{corrn0un0}.
\end{proof}
Thus all the scaling functions in the $C^0$ MRA in Theorem 5.1 of \cite{dgha} have simple explicit representations in terms of hypergeometric functions with
\begin{equation}\label{phi0}
\tilde \phi_0=\begin{cases}q_{r_0^n}(2\cdot +1)\ -1\le t<0 \\ q_{l_0^n}(2\cdot -1)\ \quad\ 0\le t<1\\ 0 \ \quad\qquad\qquad
\text{elsewhere},\end{cases}
\end{equation}
\begin{equation}\label{phi1}
\tilde\phi_1=\begin{cases}w_n(2\cdot -1) \ 0\le t<1 \\  0 \ \qquad\qquad \text{elsewhere}, \end{cases}.
\end{equation}
and for $j=2,\ldots,n$ 
\begin{equation}\label{phij}
\tilde\phi_j=\begin{cases}\phi^0_j(2\cdot -1)\  0\le t<1 \\  0\qquad\qquad \text{ elsewhere}.\end{cases}
\end{equation}

Note that for $n=4$ we find that $\alpha(4)=0$ so that only $u_{4,2}$ is needed to build $w_4(t)$ and there is no square root. Thus all the coefficients in the scaling functions will be rational numbers. It turns out that this can be generalized which is accomplished by considering an extension of the $w_n$ discussed above. To this end set
\begin{equation}\label{genw}
w_{n, m, m_1}=\alpha_{m,m_1}(n) u_{n,2m}+u_{n,2m_1},
\end{equation}
then an orthogonal intertwining MRA can be constructed if there is a nonzero $w_{n,m,m_1}$ such that $\langle (I-P_{w_{n,m,m_1}})r^{n}_0, (I-P_{w_{n,m,m_1}})l^{n}_0\rangle=0$
where $P_{w_{n,m,m_1}}$ is the orthogonal projection onto the space spanned by $w_{n,m,m_1}$. This gives the equation
\begin{equation}\label{newint}
\langle r^{n}_0, l^{n}_0\rangle\langle w_{n,m,m_1},w_{n,m,m_1}\rangle=\langle r^{n}_0, w_{n,m,m_1}\rangle\langle w_{n,m,m_1}, l^{n}_0\rangle
\end{equation}
which can be written as
\begin{align}\label{rmra}
  0 &=
    \left|
      \begin{matrix}\langle u_{n,2m},u_{n,2m}\rangle & \langle r^{n}_0,u_{n,2m}\rangle \\
              \langle r^{n}_0,u_{n,2m}\rangle & |\langle r^{n}_0,l^{n}_0\rangle|\end{matrix}\right| \alpha_{m,m_1}(n)^2\nonumber\\&
    + 2\left|
      \begin{matrix}\langle u_{n,2m},u_{n,2m_1}\rangle & \langle r^{n}_0,u_{n,2m_1}\rangle \\
              \langle r^{n}_0,u_{n,2m}\rangle & |\langle r^{n}_0,l^{n}_0\rangle|\end{matrix}\right| \alpha_{m,m_1}(n) \nonumber\\&+\left|
      \begin{matrix}\langle u_{n,2m_1},u_{n,2m_1}\rangle & \langle r^{n}_0,u_{n,2m_1}\rangle \\
              \langle r^{n}_0,u_{n,2m_1}\rangle & |\langle r^{n}_0,l^{n}_0\rangle|\end{matrix}\right|,
 \end{align}
where  equation~\eqref{roro} and second equality in equation~\eqref{rn0un0} have been used.
 Corollary~\ref{corrn0un0} shows that
\begin{align}\label{u00}
&\langle u_{n,2m_1},u_{n,2m_1}\rangle  |\langle r^{n}_0,l^{n}_0\rangle| -\langle r^{n}_0,u_{n,2m_1}\rangle^2\nonumber\\&= \langle r^{n}_0,u_{n,2m_1}\rangle^2\left(\frac{n^2+5n+2-8m_1}{(2n+1-4m_1)(n+2)}-1\right)\nonumber\\&=-\langle r^{n}_0,u_{n,2m_1}\rangle^2\frac{n(n-4m_1)}{(n+2)(2n+1-4m_1)}.
\end{align}
This gives,
\begin{lemma}\label{al0}
The function $\alpha_{0,n}(4n)$ can be chosen to be equal to zero for $n\ge 1$ so that for $-1<t<1$
$$
w_{4n,2n}(t)=u_{4n,2n}=-|t|t^{2n-1}+f^1_{n-1,2n-1}(t).
$$
\end{lemma}
Using  Corollary~\ref{corrn0un0}  to evaluate
$
\frac{\langle r_0^{4n} ,u_{4n,2n}\rangle}{\langle u_{4n,2n} u_{4n,2n}\rangle}, n\ge1
$
let
\begin{align*}
&q_{l_0^{4n}}(-t)=q_{r_0^{4n}}(t)=(I-P_{w_{4n,2n}})r_0^{4n}(t)\nonumber\\&=\frac{2(8n+1)!!}{(4n+2)!}(1+t)p^{1,2}_{4n-1}(t)-\frac{(-1)^n2^{4n}n!(3n)!}{(2n+1)!(2n)!}w_{4n,2n}(t).
\end{align*}
With  $j=2,\ldots,4n$ and $n\ge1$, set
$$
\tilde\phi_j=\begin{cases}\phi^0_j(2\cdot -1)\  0\le t<1 \\  0\qquad\qquad \text{ elsewhere},\end{cases}
$$
$$
\quad\tilde\phi_1=\begin{cases}w_{4n,2n}(2\cdot -1) \ 0\le t<1 \\  0 \ \qquad\qquad \text{elsewhere}, \end{cases}
$$
$$
\quad\tilde \phi_0=\begin{cases}q_{r_0^{4n}}(2\cdot +1)\ -1\le t<0 \\ q_{l_0^{4n}}(2\cdot -1)\ \quad\ 0\le t<1\\ 0 \ \quad\qquad\qquad
\text{elsewhere}.\end{cases}
$$

Take
$\Phi_{4n} = (\tilde\phi_0,\ldots ,\tilde\phi_{4n})^T$ with  $V_0(\Phi_{4n})$ and $V_p$ be given by equation~\eqref{multiphi}.
Then:
\begin{theorem}\label{tv1}
For $n\ge 1$,
$\Phi_{4n} $ generates an
orthogonal multiresolution analysis $\{V_p\}_{p=-\infty}^{\infty}$.
\end{theorem}
We finish with another multiresolution analysis having generators with rational coefficients.  Equation~\eqref{rmra} implies that $\alpha_{m,m_1}(n)$ will be rational if and only if
$b^2_{m,m_1}-a_m a_{m_1}=t^2$ where $t$ is a rational number. Here 
$$
b_{m,m_1}= \langle u_{n,2m},u_{n,2m_1}\rangle  |\langle r^{n}_0,l^{n}_0\rangle| -\langle r^{n}_0,u_{n,2m}\rangle\langle r^{n}_0,u_{n,2m_1}\rangle,
$$
and
$$a_m= \langle u_{n,2m},u_{n,2m}\rangle  |\langle r^{n}_0,l^{n}_0\rangle| -\langle r^{n}_0,u_{n,2m}\rangle^2.
$$
With the use of Corollary~\ref{corrn0un0}  we find,
$$
\tilde b_{m,m_1}=\frac{b_{m,m_1}}{\langle r^n_0,u_{n,2m}\rangle \langle r^n_0,u_{n,2m_1}\rangle}=-\frac{n(n-2(m+ m_1))}{(n + 2)(2n+1 - 2(m + m_1))}
$$
and from equation~\eqref{u00} above
$$
\tilde a_m=\frac{a_m}{\langle r^n_0,u_{n,2m}\rangle^2}=-\frac{n(n-4m)}{(n+2)(2n+1-4m)}
$$
which gives
$$
\tilde b^2_{m,m_1}-\tilde a_m \tilde a_{m_1}=\frac{4(m-m_1)^2 n^2(n+1)(3n+1-4(m+m_1))}{(n+2)^2(2n+1-2(m+m_1))^2(2n+1-4m)(2n+1-4m_1)}.
$$
Thus we find that admissible solutions are among those that give \break $\frac{(n+1)(3n+1-4m-4m_1)}{(2n+1-4m)(2n+1-4m_1)}=t^2$ where $t$ is a rational number. The case when $t=0$ requires $3n+1-4(m+m_1)=0$ and a set of integers which also satisfies the constraints on $u_{n,2m}$ is $n\to 4n+1,\ m\to n+1, m_1\to 2n$.
In this case 
\begin{align*}
\alpha_{n+1,2n}(4n+1)&=-b_{n+1,2n}(4n+1)/a_{n+1}(4n+1)\\&=(-1)^n\frac{(4n-1)(4n)!(n+1)!(3n)!}{3(2n)!(2n+1)!(2n+2)!(2n-1)!}.
\end{align*}
Thus:
\begin{lemma}
For $n>1$,
\begin{align*}
w_{4n+1, n+1, 2n}&=\alpha_{n+1,2n}(4n+1) u_{4n+1,2(n+1)}+u_{4n+1,4n}\\&
=-\alpha_{n+1,2n}(4n+1)|t|t^{2n-2}-|t|+\alpha_{n+1,2n}(4n+1)f^0_{n-1,2n}+f^0_{0,2n}.
\end{align*}
\end{lemma}
With the use of  Corollary~\ref{corrn0un0}  we find
$$
\langle w_{4n+1, n+1, 2n},w_{4n+1, n+1, 2n}\rangle=2^{3 - 8n}\frac{((4n)!)^3(n - 1)^2(4n + 3)}{9((2n)!)^2((2n + 1)!)^2(4n + 2)!}
$$
and
$$
\langle r^{4n+1}_0,w_{4n+1, n+1, 2n}\rangle=-2^{3 - 4n}\frac{((4n)!)^2(n - 1)}{3(2n)!(2n + 1)!(4n + 2)!}.
$$
Let
\begin{align*}
&q_{l_0^{4n+1}}(-t)=q_{r_0^{4n+1}}(t)=(I-P_{w_{4n+1, n+1, 2n}})r^{4n+1}_0(t)\nonumber\\&=r^{4n+1}_0(t)+3\frac{16^n(2n)!(2n + 1)!}{(4n+3)( n -1)(4n)!}w_{4n+1, n+1, 2n}(t) \quad n>1.
\end{align*}

For  $j=2,\ldots,4n+1$,  $n>1$ set
$$
\tilde\phi_j=\begin{cases}\phi^0_j(2\cdot -1)\  0\le t<1 \\  0\qquad\qquad \text{ elsewhere},\end{cases}
$$
$$
\qquad\qquad\tilde\phi_1=\begin{cases}w_{4n+1,n+1,2n}(2\cdot -1) \ 0\le t<1 \\  0 \ \qquad\qquad \text{elsewhere}, \end{cases}
$$$$
\qquad\tilde \phi_0=\begin{cases}q_{r^{4n+1}_0}(2\cdot +1)\ -1\le t<0 \\  q_{l^{4n+1}_0}(2\cdot -1)\ \quad\ 0\le t<1\\ 0 \ \quad\qquad\qquad
\text{elsewhere}.\end{cases}
$$
Take
$\Phi_{4n+1} = (\tilde\phi_0,\ldots ,\tilde\phi_{4n+1})^T$ with  $V_0(\Phi_{4n+1})$ and $V_p$ be given by equation~\eqref{multiphi}.
Then the above results imply,
\begin{theorem}\label{tv2}
For $n> 1$,
$\Phi_{4n+1}$ generates an
orthogonal multiresolution analysis $\{V_p\}_{p=-\infty}^{\infty}$.
\end{theorem}
\begin{remark}\label{norat}
It is interesting to note that even if we replace equation~\eqref{genw} to its most general
polynomial of degree at most $n$, there are values of $n$, i.e. $n=6$ for which  we cannot find an appropriate
polynomial $w$ with rational coefficients which leads to an orthogonal
continuous MRA.
\end{remark}
\section{Fourier transform}
The Fourier Transform of the scaling functions plays a useful role in wavelet theory \cite{da} and with the above representations we can give explicit formulas of the transform for the above functions.
We begin with
\begin{align}\label{ftpn}
&\hat \phi_{2n+\epsilon}(w)=\int_{-1}^{1}e^{-iwt}\phi_{2n+\epsilon}(t)dt=\int_0^1(e^{-iwt}+(-1)^{\epsilon}e^{iwt})\phi_{2n+\epsilon}(t)dt\nonumber\\&=h_n\int_0^1(e^{-iwt}+(-1)^{\epsilon}e^{iwt})t^{\epsilon}(1-t^2)\hypergeom21{-n+1,n+3/2+\epsilon}{1/2+\epsilon}{t^2}dt,
\end{align}
where equations~\eqref{altdefe} and \eqref{phip} have been used to obtain the last equation. Here 
$$
h_n = \frac{(-1)^{n-1}(\frac{1}{2}+\epsilon)_{n-1}}{(n+\frac{3}{2}+\epsilon)_{n-1}}.
$$
Expand the integral in the above equation to obtain
\begin{align*}
&I=h_n \int_0^1(e^{-iwt}+(-1)^{\epsilon}e^{iwt})t^{\epsilon}(1-t^2)\hypergeom21{-n+1,n+3/2+\epsilon}{1/2+\epsilon}{t^2}dt\\&=(-i)^{\epsilon}h_n \sum_{k=0}^{\infty}\frac{(-1)^kw^{2k+\epsilon}}{(2k+\epsilon)!(k+1/2+\epsilon)(k+3/2+\epsilon)}\\&\times\hypergeom32{-n+1,n+3/2+\epsilon,k+1/2+\epsilon}{1/2+\epsilon,k+5/2+\epsilon}{1}.
\end{align*}
The use of Lemma~\ref{hyp1} gives
\begin{align*}
&\hypergeom32{-n+1, n + 3/2 + \epsilon, k+\frac{1}{2}+\epsilon}{\frac{1}{2}+\epsilon, k+ \frac{5}{2}+\epsilon}{1}\\& = (-1)^{n + 1}(n-1)!(-\frac{1}{2} + \epsilon)\bigg(\frac{(k + \frac{3}{2} + \epsilon)(-\frac{3}{2} + \epsilon)(k + \frac{1}{2} +\epsilon)}{(-\frac{3}{2} + \epsilon)_{n + 3}}\\&- \frac{(k + \frac{3}{2} + \epsilon)(-\frac{3}{2} + \epsilon)(k + \frac{1}{2} + \epsilon)(-k - 1)_{n + 1}}{(-\frac{3}{2} + \epsilon)_{n + 3}(-n - k - \frac{1}{2} - \epsilon)_{n + 1}} - \frac{(n + 1)(k + \frac{3}{2} + \epsilon)(-k)_n}{(n + \frac{1}{2} + \epsilon)(-\frac{1}{2} + \epsilon)_n(-n - k - \frac{1}{2} - \epsilon)_n}\bigg).
\end{align*}
With the above formula write,
$$
I=I_1+I_2+I_3,
$$
where
$$
I_1=\frac{(-i)^{\epsilon}(n-1)!}{(n-\frac{1}{2}+\epsilon)_{n+1}}\sum_{k=0}^{\infty}\frac{(-1)^kw^{2k + \epsilon}}{(2k + \epsilon)!}=\frac{(-i)^{\epsilon}(n-1)!}{(n-\frac{1}{2}+\epsilon)_{n+1}}\cos(w-\epsilon\frac{\pi}{2}),
$$
$$
I_2=(-1)^{n-1}\frac{(-i)^{\epsilon}(n-1)!}{(n-\frac{1}{2}+\epsilon)_{n+1}}\sum_{k=0}^{\infty}\frac{(-1)^kw^{2k + 2n +  \epsilon}(-k - n - 1)_{n + 1}}{(2k + 2n + \epsilon)!(-2n - k - 1/2 - \epsilon)_{n + 1}},
$$
and
$$
I_3=(-1)^{n-1}\frac{(-i)^{\epsilon}(n-1)!(n + 1)}{(n+\frac{1}{2}+\epsilon)_{n}}\sum_{k=0}^{\infty}\frac{(-1)^kw^{2k +2n+ \epsilon}(-k-n)_n}{(2k + 2n+\epsilon)!(k + n+1/2 + \epsilon)(-2n - k - 1/2 - \epsilon)_ n}.
$$
Note that,
$$
\frac{(-k - n - 1)_{n + 1}}{(-2n - k - \frac{1}{2} - \epsilon)_{n + 1}}=\frac{(n + 1)!(n + 2)_k(n + \frac{1}{2} + \epsilon)_k(\frac{3}{2} + \epsilon)_{n-1}}{k!(2n + \frac{3}{2} + \epsilon)_k(\frac{3}{2} + \epsilon)_{2n}},
$$
$$
\frac{(-k - n )_ n}{(-2n - k - \frac{1}{2} - \epsilon)_n} =\frac{n!(n + 1)_k(n + \frac{3}{2} + \epsilon)_k(\frac{3}{2} + \epsilon)_{n}}{k!(2n + \frac{3}{2} + \epsilon)_k(3/2 + \epsilon)_{2n}},
$$
and 
$$
(2k+2n+\epsilon)!=2^{2k+2n+\epsilon}(n+1+\frac{\epsilon}{2})_k(n+\frac{1}{2}+\frac{\epsilon}{2})_k(1+\frac{\epsilon}{2})_n(\frac{1}{2}+\frac{\epsilon}{2})_n(\frac{1}{2})_{\epsilon}.
$$
Taking into account that $\epsilon\in\{0,1\},$ the above sums can be written as
\begin{align*}
I_2&=\frac{(-i)^{\epsilon}(-1)^{n-1}(n+1)(n-1)!}{(n-1/2 + \epsilon)_{n + 1}(\frac{1}{2})_{2n+1+\epsilon}}(\frac{w}{2})^{2n +  \epsilon}\sum_{k=0}^{\infty}\frac{(-1)^k(\frac{w}{2})^{2k}( n +2)_k}{k!(n+1)_k(2n +\frac{3}{2}+ \epsilon)_k}\\&=\frac{(-i)^{\epsilon}(-1)^{n-1}(n+1)(n-1)!}{(n-1/2 + \epsilon)_{n + 1}(\frac{1}{2})_{2n+1+\epsilon}}(\frac{w}{2})^{2n +  \epsilon}\\&\times\left(\hypergeom01{}{2n+\frac{3}{2}+\epsilon}{-(\frac{w}{2})^2}-\frac{(\frac{w}{2})^2}{(n+1)(2n+\frac{3}{2}+\epsilon)}\hypergeom01{}{2n+\frac{5}{2}+\epsilon}{-(\frac{w}{2})^2}\right)
\end{align*}
and
\begin{align*}
I_3&=\frac{(-i)^{\epsilon}(-1)^{n-1}(n-1)!(n + 1)}{(n+\frac{1}{2}+\epsilon)_n(\frac{1}{2})_{2n+1+\epsilon}}(\frac{w}{2})^{2n+\epsilon}\sum_{k=0}^{\infty}\frac{(-1)^k(\frac{w}{2})^{2k }}{k!(2n+\frac{3}{2}+\epsilon)_k}\\&=\frac{(-i)^{\epsilon}(-1)^{n-1}(n-1)!(n + 1)}{(n+\frac{1}{2}+\epsilon)_n(\frac{1}{2})_{2n+1+\epsilon}}(\frac{w}{2})^{2n+\epsilon}\hypergeom01{}{2n+\frac{3}{2}+\epsilon}{-(\frac{w}{2})^2}.
\end{align*}
Therefore
\begin{align}\label{ftpn1}
&\hat \phi_{2n+\epsilon}(w)\nonumber\\&=\frac{(-i)^{\epsilon}(n-1)!}{(n-\frac{1}{2}+\epsilon)_{n+1}}\cos(w-\epsilon\frac{\pi}{2})\nonumber\\&+\frac{(-1)^{n}(-i)^{\epsilon}(n-1)!}{(n-\frac{1}{2}+\epsilon)_{n+1}(\frac{1}{2})_{2n+2+\epsilon}}(\frac{w}{2})^{2n + 2+ \epsilon}\hypergeom01{}{2n+\frac{5}{2}+\epsilon}{-(\frac{w}{2})^2}\nonumber\\&+\frac{(-1)^{n-1}(-i)^{\epsilon}(n-1)!(n+1)(n+\frac{1}{2}+\epsilon)}{(n-\frac{1}{2}+\epsilon)_{n+1}(\frac{1}{2})_{2n+1+\epsilon}}(\frac{w}{2})^{2n +  \epsilon}\hypergeom01{}{2n+\frac{3}{2}+\epsilon}{-(\frac{w}{2})^2}.
\end{align}

Next we investigate
\begin{align}\label{ftln}
&\hat l^n_0(w)=\int_{-1}^1 e^{-iwt} l^n_0(t)dt=\frac{(-1)^{n-1}2(2n+1)!!}{(n+2)!}\int_{-1}^1 e^{-iwt} (1-t) p^{2,1}_{n-1}(t)dt\nonumber\\&=\frac{(n+1)}{2}\hat A(w),
\end{align}
where
$$
\hat A(w)=(-1)^{n-1}\int_{-1}^1 e^{-iwt} (1-t)\hypergeom21{-n+1,n+3}{3}{\frac{1-t}{2}}dt.
$$
Straightforward computations lead to
\begin{align*}
&\hat A(w)=4(-1)^{n-1}e^{-iw}\sum_{k=0}^{\infty}\frac{(2iw)^k}{k!}\sum_{j=0}^{n-1}\frac{(-n+1)_j(n+3)_j}{j!(3)_j}\frac{1}{k+j+2}
\\&=4(-1)^{n-1}e^{-iw}\sum_{k=0}^{\infty}\frac{(2iw)^k}{k!(k+2)}\hypergeom32{-n+1,n+3,k+2}{3,k+3}{1}.
\end{align*}
The use of Lemma~\ref{hyp1}  yields
$$
\frac{1}{k+2}\hypergeom32{-n+1,n+3,k+2}{3,k+3}{1}=-\frac{2}{n(n+2)}\left(\frac{(-1)^n}{n+1}-(-1)^n\frac{(k-n)_n}{(n+1)(k+2)_n}\right),
$$
so that
\begin{align}\label{ftln1}
\hat A(w)&=\frac{8e^{iw}}{n(n+1)(n+2)}-\frac{8e^{-iw}}{n(n+1)(n+2)}\sum_{k=0}^{\infty}\frac{(2iw)^k}{k!}\frac{(k-n)_n}{(k+2)_n}\nonumber\\&=\frac{8e^{iw}}{n(n+1)(n+2)}-\frac{8(-1)^ne^{-iw}}{n(n+1)^2(n+2)}\nonumber\\&-\frac{8e^{-iw}}{n(n+1)(n+2)}\sum_{k=1}^{\infty}\frac{(2iw)^k}{k!}\frac{(k-n)_n}{(k+2)_n}.
\end{align}
Beginning the above sum at $k=n+1$ then changing variables gives
\begin{align*}
\sum_{k=1}^{\infty}\frac{(2iw)^k}{k!}\frac{(k-n)_n}{(k+2)_n}&=\frac{(n+2)!}{(n+1)(2n+2)!}(2iw)^{n+1}\sum_{k=0}^{\infty}\frac{(n+1)_k(n+3)_k}{(n+2)_k(2n+3)_kk!}(2iw)^k\\&=\frac{(n+2)!}{(n+1)(2n+2)!}(2iw)^{n+1}\hypergeom22{n+1,n+3}{n+2,2n+3}{2iw}.
\end{align*}
Substitute this  into equation~\eqref{ftln1} then multiply by $\frac{n+1}{2}$ to find
\begin{align}\label{ftln2}
\hat l^n_0(w)&=\frac{4}{n(n+2)}e^{iw}-\frac{4(-1)^ne^{-iw}}{n(n+1)(n+2)}\nonumber\\&-\frac{4(n-1)!(2iw)^{n+1}e^{-iw}}{(2n+2)!}\hypergeom22{n+1,n+3}{n+2, 2n+3}{2iw}\nonumber\\&=\frac{4}{n(n+2)}e^{iw}-\frac{4(-1)^ne^{-iw}}{n(n+1)(n+2)}\\&-\frac{4(n-1)!(2iw)^{n+1}}{(2n+2)!}\bigg(\frac{iwn}{(n + 2)(2n + 3)}
    \hypergeom01{}{n + 5/2}{-(\frac{w}{2})^2}+ \hypergeom01{}{n + 3/2}{-(\frac{w}{2})^2}\bigg)\nonumber.
\end{align}
Here the fact that
\begin{align*}
\hypergeom22{n + 1, n + 3}{n + 2, 2n + 3}{2iw}& = e^{iw}\bigg(\frac{iwn}{(n + 2)(2n + 3)}\hypergeom01{}{n + 5/2}{-(\frac{w}{2})^2} \nonumber\\&+ \hypergeom01{}{n + 3/2}{-(\frac{w}{2})^2}\bigg)
\end{align*}
has been used to obtain the last equation.

Finally  consider,
\begin{align}\label{ftu}
\hat u_{m,n,\epsilon}(w)&=\int_{-1}^1 e^{-iwt}u_{2n+1+\epsilon,2n-2m}(t)dt\nonumber\\&=\int_{-1}^1 e^{-iwt}(-|t|t^{2m+\epsilon}+f^{\epsilon}_{m,n}(t))dt\nonumber\\&=\int_{0}^1 (e^{-iwt}+(-1)^{\epsilon}e^{iwt})(-t^{2m+1+\epsilon}+f^{\epsilon}_{m,n}(t))=B_1+B_2.
\end{align}
The first integral can be evaluated as
$$
B_1=-(-i)^{\epsilon}\sum_{k=0}^{\infty}(-1)^k\frac{w^{2k+\epsilon}}{(2k+\epsilon)!}\frac{1}{k+m+\epsilon+1}=-\frac{(-iw)^{\epsilon}}{m+\epsilon+1}\hypergeom12{m+1+\epsilon}{m+\epsilon+2,\frac{1}{2}+\epsilon}{-(\frac{w}{2})^2}.
$$
The second integral can be written as
\begin{align*}
B_2&=\frac{(\frac{1}{2}+\epsilon)_{n+1}(-m+\frac{1}{2})_n}{(m+1+\epsilon)_{n+1} n!}\int_{0}^1 (e^{-iwt}+(-1)^{\epsilon}e^{iwt})t^{\epsilon}\hypergeom32{-n,-m-\frac{1}{2},n+\frac{3}{2}+\epsilon}{\frac{1}{2}+\epsilon,-m+\frac{1}{2}}{t^2}dt\\&-\frac{(\frac{1}{2}+\epsilon)_{n}(-m-\frac{1}{2})_{n+1}}{(m+1+\epsilon)_{n+1} (n+1)!}\int_{0}^1 (e^{-iwt}+(-1)^{\epsilon}e^{iwt})t^{\epsilon}\hypergeom21{-n,n+\frac{3}{2}+\epsilon}{\frac{1}{2}+\epsilon}{t^2}dt.
\end{align*}
Write
\begin{align*}
&B_{2,1}=\int_{0}^1 (e^{-iwt}+(-1)^{\epsilon}e^{iwt})t^{\epsilon}\hypergeom32{-n,-m-\frac{1}{2},n+\frac{3}{2}+\epsilon}{\frac{1}{2}+\epsilon,-m+\frac{1}{2}}{t^2}dt\\&=(-i)^{\epsilon}\sum_{k=0}^{\infty}\frac{(-1)^kw^{2k+\epsilon}}{(2k+\epsilon)!}\sum_{j=0}^n\frac{(-n)_j(-m-\frac{1}{2})_j(n+\frac{3}{2}+\epsilon)_j}{j!(\frac{1}{2}+\epsilon)_j(-m+\frac{1}{2})_j}\frac{1}{k+j+\epsilon+1/2}\\&=(-i)^{\epsilon}\sum_{k=0}^{\infty}\frac{(-1)^kw^{2k+\epsilon}}{(2k+\epsilon)!}\frac{(-1)^{n+1}n!(m+\frac{1}{2})}{(k+m+\epsilon+1)(\frac{1}{2}+\epsilon)_{n+1}}\left(\frac{(-k)_{n+1}}{(-k-n-\frac{1}{2}-\epsilon)_{n+1}}-\frac{(m+1+\epsilon)_{n+1}}{(-n+m+\frac{1}{2})_{n+1}}\right),
\end{align*}
where Lemma~\ref{hyp1}  has been used to obtain the last equation.
Therefore after making the change of variables $k\to k+n+1$ in the first sum above then simplifying  utilizing the fact that $\epsilon\in\{0,1\}$ yields
\begin{align*}
&B_{2,1}\\&=(-iw)^{\epsilon}\frac{n!(m + 1/2)}{(n + m + \epsilon + 2)(1/2 + \epsilon)_{n + 1}(1/2 + \epsilon)_{2n + 2}}(\frac{w}{2})^{2n + 2 }\hypergeom12{n + m + \epsilon + 2}{n + m + \epsilon + 3, 2n + \frac{5}{2} + \epsilon}{-(\frac{w}{2})^2}\\& - (-iw)^{\epsilon}(-1)^{n + 1}\frac{n!(m + 2 + \epsilon)_n}{(1/2 + \epsilon)_{n + 1}(m - n + \frac{1}{2})_n}\hypergeom12{m + \epsilon + 1}{m + \epsilon + 2, 1/2 + \epsilon}{-(\frac{w}{2})^2}.
\end{align*}
Similar manipulations on the second integral in $B_2$ yield
\begin{align*}
&\int_{0}^1 (e^{-iwt}+(-1)^{\epsilon}e^{iwt})t^{\epsilon}\hypergeom21{-n,n+\frac{3}{2}+\epsilon}{\frac{1}{2}+\epsilon}{t^2}dt\\&=(-i)^{\epsilon}\sum_{k=0}^{\infty}\frac{(-1)^kw^{2k+\epsilon}}{(2k+\epsilon)!(k + \epsilon + 1/2)}\hypergeom32{-n, n+\frac{3}{2}+\epsilon, k+\epsilon+\frac{1}{2}} {\frac{1}{2}+\epsilon, k+\epsilon+\frac{3}{2}}{1}\\&=(-iw)^{\epsilon}\frac{(-1)^nn!}{(1/2 + \epsilon)_{n+1}}\hypergeom01{}{\frac{1}{2}+\epsilon}{-(w/2)^2}\\&+ (-iw)^{\epsilon}\frac{n!}{(\frac{1}{2}+\epsilon)_{n + 1}(\frac{1}{2}+\epsilon)_{2n+2}}(\frac{w}{2})^{2n + 2}\hypergeom01{}{2n+\frac{5}{2}+\epsilon}{-(\frac{w}{2})^2}.
  \end{align*}
  Thus 
  \begin{align}\label{ftuf}
&\hat u_{m,n,\epsilon}(w)\nonumber\\&=-(-iw)^{\epsilon}\frac{(-m-\frac{1}{2})_{n+1}}{(m + \epsilon + 1)_{n+2}(1/2 + \epsilon)_{2n + 2}}(\frac{w}{2})^{2n + 2 }\hypergeom12{n + m + \epsilon + 2}{n + m + \epsilon + 3, 2n + \frac{5}{2} + \epsilon}{-(\frac{w}{2})^2}\nonumber\\& -(-iw)^{\epsilon}\frac{(-1)^n(-m-1/2)_{n+1}}{(m+1+\epsilon)_{n+1}(n+1)(\frac{1}{2}+\epsilon+n)}\hypergeom01{}{\frac{1}{2}+\epsilon}{-(w/2)^2}\nonumber\\&- (-iw)^{\epsilon}\frac{(-m-\frac{1}{2})_{n+1}}{(n+1)(\frac{1}{2}+\epsilon+n)(m+1+\epsilon)_{n+1}(\frac{1}{2}+\epsilon)_{2n+2}}(\frac{w}{2})^{2n + 2}\hypergeom01{}{2n+\frac{5}{2}+\epsilon}{-(\frac{w}{2})^2}.
\end{align}
We summarize the above calculations in the following Theorem.
\begin{theorem} Let $\hat \phi_{2n+\epsilon}(w),\ \hat l^n_0(w)$, and $\hat u_{m,n,\epsilon}$ be the Fourier transforms of $\phi_{2n+\epsilon},\ l^n_0$, and $u_{2n+1+\epsilon,2n-2m}$ respectively. Then $\hat \phi_{2n+\epsilon}(w)$ is given by equation~\eqref{ftpn1}, $\hat l^n_0(w)$ is given by equation~\eqref{ftln2}, and $\hat u_{m,n,\epsilon}$ is given by equation~\eqref{ftuf}.
\end{theorem}
Note that the above ${}_0 F_1$s can be written in terms of Bessel functions of half integer order \cite[4.5.2]{aar}.
\section{ Refinement coefficients}
As noted in the introduction the scaling vector $\Phi_n$ satisfies the refinement equation
\begin{equation}\label{refine1}
\Phi_n(t)=\sum_{i=-2}^1C_i(n)\Phi_n(2t-i),
\end{equation}
where $C_i(n)$ is an $(n+1)\times (n+1)$ matrix given by
$$
C_i(n)=\tilde C_i(n)\langle \Phi_n(2\cdot) ,\Phi_n(2\cdot)\rangle^{-1},
$$
with $\tilde C_i(n)$ the $(n+1)\times (n+1)$ matrix given by
\begin{equation}\label{refine3}
\tilde C_i(n)=\int_{-\infty}^{\infty} \Phi_n(t) \Phi_n(2t-i)^Tdt.
\end{equation}
From Theorem~4.5 in \cite{dgha} we have that there is a wavelet vector function $\Psi_n$ supported in $[-1,1]$ such that,
\begin{equation}\label{psi1}
\Psi_n(t)=\begin{pmatrix}\psi_0(t)\\ \vdots\\ \psi_n(t)\end{pmatrix}=\sum_{i=-2}^1D_i(n)\Phi_n(2t-i).
\end{equation}
Here
$D_i(n)$ is an $(n+1)\times (n+1)$ matrix given by
$$
D_i(n)=\tilde D_i(n)\langle \Phi_n(2\cdot) ,\Phi_n(2\cdot)\rangle^{-1},
$$
with
\begin{equation}\label{drefine2}
\tilde D_i(n)=\int_{-\infty}^{\infty}\Psi_n(t)\Phi_n(2t-i)^Tdt.
\end{equation}
The Lemma below discusses the symmetry properties of the scaling function which will be used to determine some properties of the entries in the $C_i(n)$ and $D_i(n)$ matrices.
We begin with 
\begin{lemma}\label{smyphi}
The $\phi$'s in equations~\eqref{phi0}, \eqref{phi1}  and \eqref{phij} have the following symmetries,
\begin{equation}\label{phisym3}
\tilde\phi_0(-t)=\tilde\phi_0(t),
\end{equation}
\begin{equation}\label{phisym2}
\tilde\phi_1(1-t)=(-1)^{n+1}\tilde\phi_1(t),
\end{equation}
and 
\begin{equation}\label{phisym1}
\tilde\phi_j(1-t)=(-1)^j\tilde\phi_j(t),\quad j=2,\ldots,n.
\end{equation}
\end{lemma}
\begin{proof}
The last equation follows from the symmetry properties of the ultraspherical polynomials i.e., $p_j^{\lambda}(-t)=(-1)^jp_j^{\lambda}(t)$ and the definition of $\phi^0_j(t),\ j=2,\ldots,n$. From equation~\eqref{symw} we see that $w_n(-t)=(-1)^{n+1} w_n(t)$ so the second equation above follows from the definition of $\tilde\phi_1^n(t)$ in equation~\eqref{phi1}. Finally  since
 \begin{align*}
 \langle w_n(t),l^n_0(t)\rangle=\langle w_n(-t),r^n_0(t)\rangle=(-1)^{n+1}\langle w_n(t),r^n_0(t)\rangle,
 \end{align*}
 it follows from the definition of $\tilde\phi^n_0$ in equation~\eqref{phi0} that if $t\in(0,1)$
 \begin{align*}
 \tilde\phi_0(-t)&=q_{r^n_0}(-2t+1)=r^n_0(-2t+1)-\frac{\langle w_n,r^n_0\rangle}{\langle w_n,w_n\rangle}w_n(-2t+1)\\&
 =l^n_0(2t-1)-\frac{\langle w_n,l^n_0\rangle}{\langle w_n,w_n\rangle}w_n(2t-1)=q_{l^n_0}(2t-1)=\tilde\phi_0(t).
 \end{align*}
 \end{proof}
The above equations imply relations among the entries in the $C_j(n)$ matrices.
\begin{theorem}\label{crefine}
Let
\begin{equation}\label{crefine1}
(\tilde C_j(n))_{i,k}=\int_{-\infty}^{\infty}\tilde\phi_i(t)\tilde\phi_k(2t-j)dt,\qquad j=-2,-1,0,1,\ i,k=0,\ldots,n,
\end{equation}
and 
$$
(C_j(n))_{i,k}=(\tilde C_j(n))_{i,k}/(\langle \Phi_n(2\cdot) ,\Phi_n(2\cdot)\rangle)_{k,k}.
$$
Then
\begin{align}\label{cm2ik}
&(C_{-2}(n))_{i,k}=\begin{cases} 0 &  i=0,\ldots,n ,\ k=0,\\
(-1)^{n+1}(C_1(n))_{i,k}&  i=0, \ k=1,\\
(-1)^k(C_1(n))_{i,k}&  i=0, \ k=2,\ldots n,\\
0&  i=1\ldots,n, \ k=1,\ldots,n,
\end{cases}
\\&(C_{-1}(n))_{i,k}=\begin{cases} (C_1(n))_{i,k} & i=0,\ k=0,\\
(-1)^{n+1}(C_0(n))_{i,k}&  i=0, \ k=1,\\
(-1)^k(C_0(n))_{i,k}& i=0, \ k=2,\ldots n,\\
0&  i=1,\ldots,n, \ k=0,\ldots,n,
\end{cases}
\\&(C_{0}(n))_{i,k}= \quad 0 \qquad\qquad\qquad\qquad\quad i=1,\ldots,n ,\ k=0,
\\&(C_{1}(n))_{i,k}=\begin{cases} (C_0(n))_{i,k} &  i=1,\ k=1,\\
(-1)^{k+n+1}(C_0(n))_{i,k}& i=1, \ k=2,\ldots,n,\\
(-1)^{i+n+1}(C_0(n))_{i,k}&  i=2,\ldots,n ,\ k=1,\\
(-1)^{i+k}(C_0(n))_{i,k}& i=2,\ldots,n, \ k=2,\ldots,n.
\end{cases}
\end{align}
\end{theorem}
\begin{proof}
To show  the first line in equation~\eqref{cm2ik} note that supp$\Phi(t)\subseteq[-1,1]$ while supp$\tilde\phi_0(2t+2)=[-\frac{3}{2},-\frac{1}{2}]$ so $(\tilde C_2(n))_{i,0}=0,\ i=0,\ldots,n$. Likewise  the last line in equation~\eqref{cm2ik} follows since supp$\tilde\phi_i(t)=[0,1]$ for $i\ge1$ while supp$\tilde\phi_j(2t+2)=[-1,-\frac{1}{2}]$ for $j\ge1$. For the second and third lines in equation~\eqref{cm2ik} we have
$$
(\tilde C_{-2}(n))_{0,k}=\int_{-1}^{-\frac{1}{2}}\tilde\phi_0(t)\tilde\phi_k(2t+2)dt=\int^{1}_{\frac{1}{2}}\tilde\phi_0(-t)\tilde\phi_k(-2t+2)dt.
$$
The second line is obtained by using equations~\eqref{phisym2} and \eqref{phisym3} while the third line follows from  equations~\eqref{phisym1} and \eqref{phisym3}. The rest of the equations in the Theorem follow from similar arguments.
\end{proof}
We also have
\begin{lemma}\label{zeroco}
For $i$ odd $i>2$ and $i< j$ 
$$
(C_0(n))_{i,j}=0.
$$
For  $i$ odd $i>2$ and $i\le n$ 
$$
(C_0(n))_{i,1}=0.
$$
Finally for $n$ even
$$
(C_0(n))_{1,1}=0.
$$
\end{lemma}
\begin{proof}
For $i$ odd and $i>2$ we have $p_{i-2}^{\frac{5}{2}}(t) = t \pi_{i-3}(t)$, where  $\pi_{i-3}(t)$ is a polynomial of degree $i-3$. Therefore 
$\tilde\phi_i(t)=4t(1-t)p_{i-2}^{\frac{5}{2}}(2t-1) =4t(1-t) (2t-1) \pi_{i-3}(2t-1)$. Since  $\tilde\phi_j(2t)=8t(1-2t)p_{j-2}^{\frac{5}{2}}(4t-1)$, substituting these formulas into equation~\eqref{crefine1} we find 
$$
(\tilde C_0(n))_{i,j}=\int_{0}^{\frac{1}{2}}\tilde\phi_i(t)\tilde\phi_j(2t)dt=\int_{0}^{\frac{1}{2}} (1-2t)^2t^2\pi_{i-2}(t)p^{\frac{5}{2}}_{j-2}(4t-1) dt=0,
$$
where $\pi_{i-2}(t)=32(t-1)\pi_{i-3}(2t-1)$ is a polynomial of degree $i-2$. To see the second equation observe that for $i$ odd
$$
\tilde\phi_i(t)=t(1-2t)\tilde{\pi}_{i-2}(t)=\sum_{j=2}^{i}c_j\tilde\phi_j(2t).
$$
The result now follows since $\tilde\phi^n_1(2t)$ is orthogonal to $A_0^{n,0}(2\cdot)$. The third equation follows since for $n$ even and $t\in[0,\frac{1}{2}]$,
$$
\tilde\phi_1(t)=t(1-2t)\hat{\pi}_{n-2}(t)=\sum_{j=2}^{n}c_j\tilde\phi_j(2t),
$$
which is orthogonal to $\phi_1(2t)$.
\end{proof}

The theory of paraunitary operators \cite{hk}, \cite{lls}, \cite{ss} implies that  $n+1$ wavelets functions are needed to generate the wavelet space and these wavelets can be chosen with support in $[-1,1]$, see \cite[Theorem~4.4]{dgha}. The computation of the wavelets associated with these scaling functions is a basis completion problem and therefore non-unique. Even if we impose symmetry conditions on the wavelets similar to those of the scaling functions the problem has an infinite number of solutions. Corollary~5.2 in \cite{dgha} says that two wavelets will be supported on $[-1,1]$, one of them can be a symmetric function and the other an antisymmetric function, and the remaining $n-1$ wavelets can be constructed with support $=[0,1]$ and $\frac{1}{2}$ as an axis of symmetry. Here we give an algorithm for the construction of these wavelets. 

Corollary 5.2 in \cite{dgha}  shows that the two orthogonal wavelets not supported in $[0,1]$ can be constructed as
$$
\tilde\psi_0=(I-P_{\tilde\phi_0})\tilde\phi_0(2\cdot),
$$
and
$$
\tilde\psi_1=(\chi_{[0,1]}-\chi_{[-1,0]})(I-P_{\tilde\phi_0(2\cdot)})\tilde\phi_0.
$$
The remaining $n-1$ wavelet functions are supported on $[0,1]$ and must be orthogonal to $\tilde\phi_0$ and $\tilde\phi_0(2\cdot)$ which makes them orthogonal to the above two wavelets.
This is a sketch of the  construction. Build the matrices $C_{0}(n), C_1(n)$ and let $E(n)$ be the $(n+1)\times(2n+2)$ matrix 
$$
E(n)=( C_0(n), C_1(n)).
$$
 To construct the $n-1$ wavelet functions supported on $[0,1]$ we begin with the ones that are even with respect to $\frac{1}{2}$. From Lemma~\ref{smyphi} the wavelets supported on $[0,1]$ that are symmetric with respect to $\frac{1}{2}$ must have the same coefficients for $\tilde \phi_{2i}(2t)$ and $\tilde\phi_{2i}(2t-1)$, $i\ge1$ while the coefficients for $\tilde \phi_{2i+1}(2t)$ and $\tilde \phi_{2i+1}(2t-1)$, $i\ge 1$ must have the same magnitude but opposite signs. Furthermore the coefficients of $\tilde\phi_1(2t)$ and $\tilde\phi_1(2t-1)$ will be of the same magnitude but with a sign of $(-1)^{n+1}$.  Any wavelet function with support in $[0,1]$ that is symmetric with respect to $\frac{1}{2}$ must have a coefficient vector $p\in\R^{2n+2}$ given by
  $$
 p=[0,b_1,\ldots, b_{n},b_0, (-1)^{n+1}b_1,b_2,-b_3,\ldots,(-1)^n b_n]^T.
 $$
 Since the entry before $b_1$ is equal to zero  the wavelets constructed in this way will be orthogonal to $\tilde\phi_0(2\cdot)$. 
  Find all the solutions to $E(n)p=0$ then substitute these solution into $p$ and call this vector $v$ which gives rise to a subspace $M$ of the null space of $E$. If the number of free variables is more than one set one of the free variables, $b_i=1$ say, and the rest equal to zero. Denote this solution after being normalized as $v_1$. Let 
 $$
 \tilde\psi_2(t)=v_1^T \left(\begin{matrix}\Phi(2t)\\ \Phi(2t-1)\end{matrix}\right)
  $$
  then it follows from equations~\eqref{refine1} and \eqref{refine3}  that 
  $$
  \int_{-1}^1\tilde\psi_2(t)\tilde\phi_j(t)dt=\int_0^1\tilde\psi_2(t)\tilde\phi_j(t)dt=0,\ j=0,\ldots,n.
  $$
  This implies that $\int_{-1}^1\tilde\psi_2(t)\tilde\psi_j(t)dt=0,\ j=0,1$.
  
  Next find the subspace of $M$ perpendicular to $v_1$ and repeat the above procedure until an orthogonal basis of $M$ is found. This basis will give an orthogonal basis of vectors in the wavelet space $W_0$ supported in $[0,1]$ each of which is symmetric to $\frac{1}{2}$. Next we find the wavelets that are antisymmetric with respect to $\frac{1}{2}$. The coefficient vector must be of the form
 $$
 p=[0,b_1,\ldots, b_{n},0, (-1)^{n}b_1,-b_2,b_3,\ldots,-(-1)^{n} b_n]^T.
 $$
 Now carry out steps similar to the symmetric case. We now have
 \begin{theorem}\label{waveident}
Suppose $\tilde\psi_j$,  $j=2,\dots,n$ are constructed using the algorithm above. Then
the set of functions $\{ \tilde\psi_j(2^k\cdot-i),i,k\in\Z, j=0,\ldots,n \}$ forms an orthogonal, compactly supported, continuous, piecewise polynomial basis for $L^2(\R)$.  
\end{theorem}

 With the above multiwavelets it is not difficult to construct a basis for $L^2([0,1])$ and this was done in \cite[Theorem~5.3]{dgha}  which with a slight change of notation will be stated for the reader's convenience.
 Let 
\begin{align*}
&W_0 =W _0(\Psi_n)=\cl{\rm span}\{\tilde\psi_i({\cdot}-j): i = 0,1,\dots,n,\ j\in\Z\}\ \text{with}\nonumber\\& W_p = \{f(2^p{\cdot}): f\in W_0\}, \quad p\in\N.
\end{align*}
 We now have\footnote{We note here that the index $i$ in \cite[Theorems~5.3 and 6.3]{dgha} is supposed to start at $0$.}
 \begin{theorem}\label{wavee}
 Suppose $\tilde\psi_j$,  $j=2,\dots,n$ are constructed using the algorithm above. Then for $k\ge0$ the set
$\{\tilde\phi_i(2^k t-j)|_{[0,1]}\st i=0, \dots,n,\ 0\le j\le 2^k-1+\delta_{0,i}\}$
is an orthogonal basis for $\bar V_k = V_k\cap L^2[0,1]$ while
$\{\tilde\psi_i(2^k t-j)|_{[0,1]}\st \ i = 0,\dots,n,
  \ \delta_{1,i}\le j \le 2^k-1+\delta_{0,i}\}$
forms an orthogonal basis for $\bar W_k = W_k\cap L^2[0,1]$.
Furthermore $\cl\bar V_0\oplus\bigoplus_{k\ge 0}\bar W_k = L^2[0,1]$.
\end{theorem}
Thus roughly speaking, in order to obtain a basis for $L^2[0,1]$ we take the nonzero restrictions of any scaling functions in $V_0$ to $[0,1]$ as well as any wavelet function whose support is totally in $[0,1]$ plus the restrictions of $\tilde\psi_0(2^k\cdot),\ k\ge0$ and $\tilde\psi_0(2^k(\cdot-1)),\ k\ge0$  to $[0,1]$.

Although the above discussion is for the MRA given by the equations~\eqref{phi0}, \eqref{phi1}, and \eqref{phij} it is straightforward to adapt the methods  to the other MRAs discussed in section 4.
\section{Acknowledgements}
JSG would like to thank Kolja Brix for pointing out errors in the graphs and tables of the scaling functions in \cite{dgha} and 
providing his computations related to the scaling functions and wavelets. He would also like to thank G. Donovan for discussions related to the wavelet construction algorithm.

The present version of the paper owes a lot to a referee whose numerous corrections helped us improve considerably the exposition.


\begin{thebibliography}{99}

\bibitem{abcr}  B.~Alpert, G.~Beylkin, R.~Coifman, and V.~Rokhlin, {\em Wavelet-like 
bases for the fast solution of second-kind integral equations}, SIAM J. Sci. Comput. {\bf 14} (1993), 159--184.

\bibitem{aar}  G.~E.~Andrews, R.~Askey, and R.~Roy, {\em Special Functions}, Encyclopedia Math. Appl. {\bf 71}  Cambridge University Press, Cambridge, UK, (1999).

\bibitem{bcd} M.~Bachmayr, A.~Cohen and W.~Dahmen, {\em Parametric PDEs: sparse or low-rank approximations},
IMA J. Numer. Anal. {\bf 38} (2018), 1661--1708.

\bibitem{ba} W.N. Bailey, {\it Generalized Hypergeometric Series}, Cambridge University Press, Cambridge, UK, (1935).

\bibitem{bcc} K.~Brix, M.~Campos Pinto, C.~Canuto, and W.~ Dahmen, {\em Multilevel preconditioning of discontinuous 
Galerkin spectral element methods. Part I: geometrically conforming meshes},
IMA J. Numer. Anal.  {\bf 35} (2015), 1487--1532.

\bibitem{cf} D.~\v{C}ern\'{a} and K.~Fi\v{n}kov\'{a}, {\em Option pricing under multifactor {B}lack-{S}choles model using orthogonal spline wavelets},
Math. Comput. Simulation {\bf 220} (2024), 309--340.

\bibitem{da} I.~Daubechies, {\em Ten Lectures on Wavelets}, CBMS-NSF Regional Series in Applied Math, {\bf 61} SIAM, Philadelphia, 1992.

\bibitem{deb} C.~de Boor, {\em Splines as linear combination of B-splines}, Approximation Theory II (G. Lorentz, C. Chui, and L. Schumaker, eds.), Academic Press, NY, (1976), 1--47.

\bibitem{ddr} C.~de Boor, R.~A.~DeVore, and A.~Ron, {\em The structure of finitely generated shift-invariant spaces in
   $L_2({\bf R}^d)$}, J. Funct. Anal. {\bf 119} (1994), 37--78.

\bibitem{dss} T.~J.~Dijkema, C.~Schwab, and R.~ Stevenson, 
{\em An adaptive wavelet method for solving high-dimensional elliptic PDEs},
Constr. Approx. {\bf 30} (2009), 423--455.


\bibitem{dgh}G.~C.~Donovan, J.~S.~Geronimo,  and D.~P.~Hardin,
  {\em Intertwining multiresolution analyses and the construction of piecewise polynomial wavelets}, SIAM J. Math. Anal.  {\bf 27} (1996), 1791--1815.

\bibitem{dgha}G.~C.~Donovan, J.~S.~Geronimo, and D.~P.~Hardin,
  {\em Orthogonal polynomials and the construction of piecewise polynomial smooth wavelets}, SIAM J. Math. Anal. {\bf 30} (1999), 1029--1056.
  
\bibitem{gi} J.~S.~Geronimo and P.~Iliev, {\em A hypergeometric basis for the Alpert multiresolution analysis}, SIAM J. Math. Anal. {\bf 47} (2015), 654--668. 

\bibitem{glv} J.~S.~Geronimo, P.~Iliev, W.~Van Assche, 
{\em Alpert multiwavelets and Legendre-Angelesco multiple orthogonal polynomials},
SIAM J. Math. Anal. {\bf 49} (2017), 626--645.

\bibitem{gm} J.~S.~Geronimo and F.~Marcell\'an, {\em On Alpert multiwavelets}, Proc. Amer. Math. Soc. {\bf 143} (2015), 2479--2494. 

\bibitem{ggl} S.~S.~ Goh, T.~N.~T.~ Goodman, and S.~L.~Lee
{\em Orthogonal polynomials, biorthogonal polynomials and spline functions},
Appl. Comput. Harmon. Anal. {\bf 52} (2021), 141--164.

\bibitem{gl}T.~N.~T.~Goodman and S.~L.~Lee, {\em Wavelets of multiplicity $r$},
 Trans. Amer. Math. Soc. {\bf 342} (1994), 307--324.
  
 \bibitem{hk} D.~P.~ Hardin and B.~ Kessler, {\em Orthogonal macroelement scaling vectors and wavelets in 1-D}, Wavelet and fractal methods in science and engineering, Part I.  Arab. J. Sci. Eng. Sect. C Theme Issues, {\bf 28} (2003), 73--88.

\bibitem{hkm} D.~P.~ Hardin, B.~ Kessler, and P.~R.~ Massopust, {\em Multiresolution analyses and fractal functions},  J. Approx. 
  Theory {\bf 71} (1992), 104--120. 
 
 \bibitem{h} L.~Herv\'{e}, {\em Multi-resolution analysis of multiplicity {$d$}: applications to dyadic interpolation}, Appl. Comput. Harmon. Anal. {\bf 1} (1994), 299--315.

\bibitem{js} R.~Jia and Z.~Shen, {\em Multiresolution analysis and Wavelets}, Proc. Edinburgh Math. Soc. {\bf 37} (1994), 271--300.

\bibitem{lls} W.~Lawton,  S.~ L.~Lee and Z.~ Shen, {\em An algorithm for matrix extension and wavelet construction},
Math. Comp. {\bf 65} (1996), 723–737.

\bibitem{ss} G.~Strang and V. Strela, {\em Short wavelets and matrix dilation equations}, IEEE Trans. Sig. Proc. {\bf 41} (1995), 108--115.

\bibitem{sz} G.~Szeg\H o, {\em Orthogonal polynomials}, AMS Colloq. Publ, {\bf 23}, AMS, Providence, RI, 1939.

\end{thebibliography}
\end{document}